\documentclass[11pt]{article}
\usepackage{amsmath,amsthm,amssymb,mathrsfs}
\usepackage{bm,pifont,xcolor, tikz}

\makeatletter
\newcommand{\rmnum}[1]{\romannumeral #1}
\newcommand{\Rmnum}[1]{\expandafter\@slowromancap\romannumeral #1@}
\makeatother

\usepackage{geometry}
\usepackage[colorlinks=true,
linkcolor=black,citecolor=black,
urlcolor=black]{hyperref}

\makeatletter
\def\@seccntDot{.}
\def\@seccntformat#1{\csname the#1\endcsname\@seccntDot\hskip 0.5em}
\renewcommand\section{\@startsection{section}{1}{\z@}%
{18\p@ \@plus 6\p@ \@minus 3\p@}%
{9\p@ \@plus 6\p@ \@minus 3\p@}%
{\large\bfseries\boldmath}}
\renewcommand\subsection{\@startsection{subsection}{2}{\z@}%
{12\p@ \@plus 6\p@ \@minus 3\p@}%
{3\p@ \@plus 6\p@ \@minus 3\p@}%
{\bfseries\boldmath}}
\renewcommand\subsubsection{\@startsection{subsubsection}{3}{\z@}%
{12\p@ \@plus 6\p@ \@minus 3\p@}%
{\p@}%
{\bfseries\boldmath}}
\makeatother

\usepackage{microtype}

\theoremstyle{plain}
\newtheorem{theorem}{Theorem}[section]
\newtheorem{lemma}{Lemma}[section]

\newtheorem{conjecture}{Conjecture}[section]

\theoremstyle{definition}
\newtheorem{definition}{Definition}[section]
\newtheorem{remark}{Remark}[section]

\newtheorem{claim}{Claim}

\numberwithin{equation}{section}
\allowdisplaybreaks
\title{Uniquely $C_{4}^{+}$-saturated graphs \thanks{Email addresses: \texttt{yuyingli1205@163.com} (Y. Li),
 \texttt{kexxu1221@126.com} (K. Xu), \texttt{gerbner.daniel@renyi.hu} (D. Gerbner), \texttt{wzhliu7502@nuaa.edu.cn} (W. Liu).}}

\author {Yuying Li $^{a}$,  ~~ Kexiang Xu $^{a}$, ~~D\'{a}niel Gerbner $^{b}$, ~~ Wenzhong Liu $^{a}$  \\[3mm]
 \small $^{a}$  School of Mathematics, Nanjing University of Aeronautics and Astronautics,\\
 \small Nanjing, Jiangsu 210016, PR China\\
\small $^b$ Alfr\'{e}d R\'{e}nyi Institute of Mathematics, HUN-REN }

\date{}

\begin{document}
\maketitle

\begin{abstract}
A  graph $G$ is  \textit{uniquely $H$-saturated} if it contains no copy of  a graph $H$ as a subgraph,  but adding any new edge into $G$ creates exactly one copy of $H$. Let $C_{4}^{+}$ be the diamond graph consisting of a $4$-cycle $C_{4}$ with one chord and  $C_{3}^{*}$  be the  graph  consisting of a triangle with a pendant edge.  In this paper  we  prove that a nontrivial uniquely $C_{4}^{+}$-saturated graph $G$  has girth $3$  or  $4$.  Further,  $G$  has  girth $4$  if and only if it  is a  strongly regular graph  with  special parameters.
For $n>18k^{2}-24k+10$ with $k\geq2$,   there are no uniquely $C_{4}^{+}$-saturated graphs  on  $n$  vertices  with  $k$ triangles.   In  particular,  $C_{3}^{*}$  is  the only nontrivial uniquely $C_{4}^{+}$-saturated graph with  one triangle,  and there are no uniquely $C_{4}^{+}$-saturated graphs with  two,  three or four triangles.

\par\vspace{2mm}

\noindent{\bfseries Keywords:   graph saturation, uniquely saturated graph, strongly regular graph}
\par\vspace{1mm}

\noindent{\bfseries  AMS Classification (2020):  05C35, 05C07}
\end{abstract}

\section{Introduction}

All graphs considered in this paper are finite,  undirected and simple.  Let $G$ be a graph and as usual, we denote by  $V(G)$,  $E(G)$,  $|G|$, $e(G)$ and $\overline{G}$ the vertex set, edge set,   the number of vertices, the number of edges and the complement of $G$, respectively. For any $v\in V(G)$, let $N_{G}(v)$ be the neighborhood of $v$ in $G$ and  $d_{G}(v)=|N_G(v)|$ be the degree of $v$ in $G$,  which  are denoted by $N(v)$  and   $d(v)$  for short  if there is no ambiguity.
Moreover, $N_{G}[v]=N_{G}(v)\cup\{v\}$  is  the  close  neighborhood of $v$ in $G$.  For a vertex $v\in V(G)$  with $A\subseteq V(G)$,  let  $d_{A}(v)$  be the number of neighbors of $v$ in  $A$.
A vertex $v$ is a  \textit{universal vertex} of  $G$ if  $d_{G}(v)=|G|-1$.  As usual, let $\Delta(G)$ and $\delta(G)$ be the maximum degree and the minimum degree, respectively,  of $G$.  The  \textit{distance $d_{G}(x,y)$}  of two vertices $x$, $y$ is the length of a shortest $(x,y)$ path in $G$.  The greatest distance between any two vertices in $G$ is the \textit{diameter} of $G$, denoted by $diam(G)$.  The \textit{girth} $g(G)$ of a graph $G$ is the minimum length of cycles in it.  The graphs with diameter $d$ and girth $2d+1$ are the \textit{Moore graphs}.
 Denote by $K_{n}$, $P_{n}$,  $S_{n}$ and $C_{n}$ the complete graph, path,  star graph and cycle on $n$ vertices, respectively.
 A \textit{diametical path} of graph $G$ is a path with length $diam(G)$  in $G$.
 The double star $D_{s, t}$  is  the graph  formed  by  attaching $s$ pendant vertices to one endpoint of $K_{2}$, and  $t$ pendant vertices  to the  other endpoint of it. The  double star  is   balanced  if  $s=t$.
For $A\subseteq V(G)$, let $G[A]$   be the subgraph of $G$ induced by $A$.  Given $A, B \subseteq V(G)$, we denote by $E_{G}[A,B]$, or $E[A,B]$ for simplicity, the  edge set between $A$ and $B$ in $G$  and  its  cardinality will be written as $e_{G}[A,B]$ or $e[A,B]$ for short.  For any edge $e\in E(\overline{G})$ (here we call $e$ a non-edge of $G$),  we  write by $G+e$ the graph obtained from $G$ by adding the new edge $e$.  Given any two  vertex-disjoint graphs $G$ and $H$,  their  \textit{union}  $G\cup H$   is the graph with vertex set $V(G)\cup V(H)$ and edge set $E(G)\cup E(H)$  and  their  \textit{join}  $G\vee H$   is the graph with vertex set $V(G)\cup V(H)$ and edge set $E(G)\cup E(H)\cup\{gh : g\in V(G), h\in V(H)\}$.  Please refer to  \cite{A.Bondy2008}  for  more definitions and  notations on graph theory.

Given a graph $H$,  a graph $G$ is \textit{$H$-free} if $G$ does not contain $H$ as a  subgraph.
A  graph $G$  is   \textit{$H$-saturated} if $G$ is $H$-free,  but the addition of any edge missing from $G$ creates a copy of $H$ in the resultant graph.  Let $G$ be an $H$-free graph with a non-edge $e$,  we  say $e$ is an \textit{$H$-saturating edge} of $G$  if $G+e$ contains a copy of $H$.
The study of $H$-saturated graphs began since Mantel \cite{W.Mantel1907} proved that $K_{\lceil\frac{n}{2}\rceil,\lfloor\frac{n}{2}\rfloor}$ is the $K_{3}$-saturated graph on  $n$  vertices with most edges.  Tur\'{a}n \cite{P.Turan1941} generalized this  result by  proving that the  $K_{r}$-saturated graphs on  $n$ vertices  with  most edges are the complete  $(r-1)$-partite graph with partite sets as balanced as possible.   In the opposite direction,  Erd\H{o}s et al. \cite{P.Erdos1964} determined that the $K_{r}$-saturated graphs on  $n$ vertices  with  fewest edges   are   $K_{r-2}\vee\overline{K_{n-r+2}}$.
 Please see  an informative survey  \cite{B.L.Currie2021}  for some detailed results about the $H$-saturated graphs  with the fewest edges.

Given a graph $H$ and an $H$-saturated graph $G$,  we say that $G$ is \textit{uniquely $H$-saturated} if the addition of any edge to $G$ creates exactly one copy of $H$.  If $H$ has $t$ vertices, every complete graph with fewer than $t$  vertices is trivially  uniquely $H$-saturated, since there are no edges to be added into it. A uniquely $H$-saturated graph is \textit{nontrivial} if it has at least  $t$ vertices.
 Cooper,  Lenz,   LeSaulnier,    Wenger,  and  West  \cite{J.Cooper2012}  initiated the study of uniquely saturated graphs by determining all uniquely $C_{k}$-saturated graphs  for  $k\in\{3, 4\}$.
 Wenger and West \cite{P.S.Wenger2016}  proved that  nontrivial  uniquely $C_{5}$-saturated graphs are precisely the friendship graphs and there are no  nontrivial  uniquely $C_{k}$-saturated graphs for $k\in\{6,7,8\}$.
For the path,  Wenger \cite{P.S.Wenger2010}  proved  that  no nontrivial uniquely $P_{k}$-saturated graphs exist  for $k\geq5$  and also characterized  nontrivial uniquely $P_{k}$-saturated graphs  for  $k\in\{2,3,4\}$.   Berman,  Chappell,  Faudree,  Gimbel and  Hartman \cite{L.W.Berman2016} determined that  the trees $T$ for which there exist an infinite number of uniquely $T$-saturated graphs are precisely  balanced double stars.   They conjectured that double stars are the only trees  $T$ for which nontrivial uniquely $T$-saturated graphs exist.
By now  few examples of uniquely $K_{r}$-saturated graphs were known, and little was known about their properties.
Hartke  and  Stolee  \cite{S.G.Hartke2012} considered uniquely $K_{r}$-saturated graphs  with  no universal vertex,  called $r$-primitive graphs.  They  found ten new $r$-primitive graphs with $4\leq r\leq7$  and  two new infinite families of uniquely $K_{r}$-saturated  Cayley graphs.    Gy\'{a}rf\'{a}s,  Hartkey   and  Viss  \cite{A.Gyarfas2018}  showed  that there exist no $r$-primitive graphs  on $n$ vertices for any   $n\leq2r-3$.

Let $C_{4}^{+}$ be the diamond graph consisting of a $4$-cycle $C_{4}$ with one chord and  $C_{3}^{*}$  be the  graph  consisting of a triangle with a pendant edge.  In this paper  we consider  uniquely  $C_{4}^{+}$-saturated graphs.   The paper is organized as follows.
In Section 2,   we give  some  structural properties  of nontrivial uniquely $C_{4}^{+}$-saturated graphs and  show  that  they  have girth $3$  or  $4$.
In Section 3,  we  prove  that $G$  is a uniquely $C_{4}^{+}$-saturated graph with girth $4$  if and only if  it  is a  strongly regular graph  with  special parameters.
For $n>18k^{2}-24k+10$  with $k\geq2$,  we get that  there are no uniquely $C_{4}^{+}$-saturated graphs  on  $n$  vertices  with  $k$ triangles.
In  particular, $C_{3}^{*}$  is the only nontrivial uniquely $C_{4}^{+}$-saturated graph with one triangle,  and  there are no uniquely $C_{4}^{+}$-saturated graphs  with  two,  three or four triangles.  We believe that there are no uniquely $C_{4}^{+}$-saturated graphs with more than four triangles, giving the following conjecture.

\begin{conjecture}
A graph $G$ is nontrivial uniquely $C_{4}^{+}$-saturated if and only if $G$ is a strongly regular graph with parameters $(n,k,0,2)$ or $G\cong C_{3}^{*}$.

\end{conjecture}

Finally, in Section 4 we make some observations on uniquely $B_p$-saturated graphs, where the \textit{book graph} $B_p$ is  obtained by joining the two vertices in the part of size 2 of $K_{2,p}$,  as a generalization of $C_{4}^{+}$.

\section{Structural Properties  of  uniquely $C_{4}^{+}$-saturated graphs}

For a $C_{4}^{+}$-saturated graph  $G$,  we begin with a classification of $C_{4}^{+}$-saturating edges  in it.  Let  $uv$  be a non-edge of $G$, then $uv$ is an \Rmnum{1}-type  $C_{4}^{+}$-saturating edge if $d(u)=d(v)=3$  in $C_{4}^{+}\subseteq G+uv$ (see Fig. 1 left),  $uv$ is a  \Rmnum{2}-type  $C_{4}^{+}$-saturating edge  if  $d(u)=3$  and  $d(v)=2$ in $C_{4}^{+}\subseteq G+uv$ (see Fig. 1 right).
Moreover,  the induced graph of $\{u,v,a,b\}$ is $C_{4}$  if $uv$ is  an  \Rmnum{1}-type  $C_{4}^{+}$-saturating edge   and  is  $C_{3}^{*}$ if $uv$ is a  \Rmnum{2}-type  $C_{4}^{+}$-saturating edge.

\begin{center}
\begin{tikzpicture} [inner sep=0.8mm]
\tikzstyle{place}=[circle,draw=black!,fill=black!]
\tikzstyle{r}=[circle,draw=red!,fill=red!]

\node at (1.2,0)    [place] {};
\node at (0,1.5)    [place] {};
\node at (2.4,1.5)    [place] {};
\node at (1.2,3)    [place] {};

\coordinate [label=above left:$u$] () at (0,1.5);
\coordinate [label=above right:$v$] () at (2.4,1.5);
\coordinate [label=below right:$b$] () at (1.2,0);
\coordinate [label=above right:$a$] () at (1.2,3);

\draw[black,line width=1] (1.2,3) -- (0,1.5);
\draw[black,line width=1] (1.2,3) -- (2.4,1.5);
\draw[black,line width=1] (1.2,0) -- (0,1.5);
\draw[black,line width=1] (1.2,0) -- (2.4,1.5);
\draw[dashed, line width=1] (0,1.5) -- (2.4,1.5);

\node at (6.2,0)    [place] {};
\node at (5,1.5)    [place] {};
\node at (7.4,1.5)    [place] {};
\node at (6.2,3)    [place] {};

\coordinate [label=above left:$u$] () at (5,1.5);
\coordinate [label=above right:$b$] () at (7.4,1.5);
\coordinate [label=below right:$v$] () at (6.2,0);
\coordinate [label=above right:$a$] () at (6.2,3);

\draw[black,line width=1] (6.2,3) -- (5,1.5);
\draw[black,line width=1] (6.2,3) -- (7.4,1.5);
\draw[black,line width=1] (6.2,0) -- (7.4,1.5);
\draw[black,line width=1] (5,1.5) -- (7.4,1.5);
\draw[dashed, line width=1] (6.2,0) -- (5,1.5);

\end{tikzpicture}

\end{center}
\begin{center}
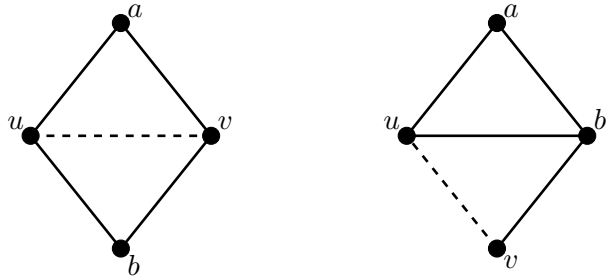
Figure 1 : \Rmnum{1}-type  $C_{4}^{+}$-saturating edge  $uv$  (left)  and  \Rmnum{2}-type  $C_{4}^{+}$-saturating edge  $uv$ (right).
\end{center}

\begin{lemma}\label{lemma:2.1}
Let $G$  be a nontrivial uniquely $C_{4}^{+}$-saturated graph, then $G$  is connected with $diam(G)=2$.
\end{lemma}
 \begin{proof}
If  $G$  is  disconnected,  let  $u$ and $v$  be  vertices in two  components of $G$, respectively.  Since  $uv$  is a cut-edge  in  $G+uv$  but  $C_{4}^{+}$  has  no  cut-edge,    there is no $C_{4}^{+}$  in $G+uv$,  a  contradiction.  If  $G$  has diameter at least 3,  let $P$  be  a diametical  path   of  $G$  with endpoints $x$  and $y$,    we  have  $|P|\geq4$.  Thus $G+xy$  has  no $C_{4}^{+}$,  a  contradiction  again.   Therefore, $G$  has  diameter  $2$.
\end{proof}

\begin{lemma}\label{lemma:2.2}
Let $G$  be a nontrivial  uniquely $C_{4}^{+}$-saturated graph  and  $u$,  $v$  be any two non-adjacent vertices  of $G$.  Then the distance of $u$ and $v$ is $2$ and  the number of  paths of length $2$ between $u$ and $v$  is $1$  or  $2$.  Moreover,    the number of  such paths is  $1$  if  and  only  if   there  exists  exactly one  $K_{3}$ sharing  a common edge with such a path.
\end{lemma}
\begin{proof}
For  any two non-adjacent vertices $u$  and  $v $ of $G$,  we know their distance is 2  by   \autoref{lemma:2.1}.  By the classification of the $C_{4}^{+}$-saturating  edges,  if  $uv$  is a \Rmnum{2}-type $C_{4}^{+}$-saturating edge,   then  there is one path of length $2$ between $u$ and $v$ and  there exists exactly one  $K_{3}$ sharing  a common edge with that path.  If  $uv$  is an \Rmnum{1}-type $C_{4}^{+}$-saturating edge,  then  the number of  paths of length $2$ between $u$ and $v$  is $2$  and  there  is no  such  $K_{3}$.
\end{proof}

Let the \textit{house graph} denote the graph  obtained by  adding a chord to $C_{5}$  and let the \textit{bowknot  graph}  denote   two triangles sharing one vertex.  Let us show that the uniquely $C_{4}^{+}$-saturated graph  has some  forbidden  subgraphs  in  the  following result.

\begin{lemma}\label{lemma:2.3}
Let $G$  be a  uniquely $C_{4}^{+}$-saturated graph, then $G$  contains no   bowknot  graph,   house graph  and  $K_{2,3}$  as  subgraphs.
\end{lemma}
 \begin{proof}
Let  $u$ and  $v$ be vertices in a copy of the bowknot graph in $G$ that are not the common vertex of the two triangles.  $G$  contains $C_{4}^{+}$  if  $u$,  $v$  are  adjacent in $G$ and $G+uv$  contains  two $C_{4}^{+}$  if  $u$,  $v$  are  non-adjacent  in $G$,  which  contradicts  the  fact that $G$  is  uniquely $C_{4}^{+}$-saturated.  Similarly,  we  consider  the  two vertices  $u$,  $v$  with  $d(u,v)=2$ in the house graph  and  consider the two vertices $u$, $v$  in the smaller part of  $K_{2,3}$,  and get a contradiction again.
\end{proof}

Let $G$  be a nontrivial uniquely $C_{4}^{+}$-saturated graph,  we next  examine  the girth of $G$.

\begin{lemma}\label{lemma:2.4}
Let $G$  be a nontrivial uniquely $C_{4}^{+}$-saturated graph, then $G$   has girth $3$ or $4$.
\end{lemma}
 \begin{proof}
If  $G$  has girth at least $6$,  then $diam(G)\geq3$  as a  contradiction  to  \autoref{lemma:2.1}.   Now we assume that $g(G)=5$  and  let $C$ be a shortest  cycle in $G$  with  $V(C)=\{v_{1},v_{2},v_{3},v_{4},v_{5}\}$.   But  there exists a $C_{4}$  in $G$   if  $v_{1}v_{3}$  is an  \Rmnum{1}-type $C_{4}^{+}$-saturating edge   and   there exists a $C_{3}^{*}$  in $G$  if  $v_{1}v_{3}$  is a  \Rmnum{2}-type $C_{4}^{+}$-saturating edge,  which contradicts   $g(G)=5$.   Thus  $G$   has girth $3$ or $4$.
\end{proof}

\section{Uniquely $C_{4}^{+}$-saturated graphs}

Firstly  we  introduce the definition of strongly regular graphs.

\begin{definition}
A  graph $G$ on $n$  vertices is a strongly regular  graph with  parameters $(n, k, \lambda, \mu)$  whenever it is not complete or edgeless and

(\rmnum{1}). each vertex has  degree  $k$;

(\rmnum{2}). each pair of adjacent vertices  of  $G$  has  $\lambda$  common  neighbors;

(\rmnum{3}). each pair of non-adjacent vertices of  $G$   has  $\mu$  common  neighbors.
\end{definition}

In particular,   except for balanced complete bipartite graph  $K_{k,k}$  with  parameters $(2k, k, 0, k)$,  there are  known only seven triangle-free strongly regular graphs   \cite{N.L.Biggs2009,A.Razborov2022}:

(\rmnum{1}). The \textit{pentagon} with parameters $(5,2,0,1)$;

(\rmnum{2}). The \textit{Petersen graph} with parameters $(10,3,0,1)$;

(\rmnum{3}). The \textit{Hoffman-Singleton graph} with parameters $(50,7,0,1)$;

(\rmnum{4}). The \textit{folded $5$-cube} with parameters $(16,5,0,2)$;

(\rmnum{5}). The \textit{Gewirtz graph}  with parameters $(56,10,0,2)$;

(\rmnum{6}). The \textit{$M_{22}$ graph}  with parameters $(77,16,0,4)$;

(\rmnum{7}). The \textit{Higman-Sims graph}, with parameters $(100,22,0,6)$.

Each of these seven graphs is uniquely determined by its parameters.  We  can  see that  the first three  graphs  are  Moore graphs with diameter two, a  more detailed description of these graphs can be found in \cite{A.E.Brouwer2012}. It is unknown whether there are any other examples  of  triangle-free strongly regular graphs.

For a  graph  $G$  and  $U\subseteq V(G)$, let $d(v,U)=min\{d(v,u): u\in U\}$,  $N(U)=\{v\in V(G): d(v,U)=1\}$ and $N^{k}(U)=\{v\in V(G): d(v,U)=k\}$.

Let $G$  be a  nontrivial  uniquely $C_{4}^{+}$-saturated graph,   by  \autoref{lemma:2.4},  we have  $g(G)=3$  or  $4$.   In this section,  we  show  that  $G$  has girth $4$  if  and  only if  it is a strongly regular graph  with  special parameters.
For $n>18k^{2}-24k+10$  with  $k\geq2$,  we  prove  that   there are no uniquely $C_{4}^{+}$-saturated graphs  on  $n$  vertices  with  $k$ triangles.  In particular,  $C_{3}^{*}$  is the only nontrivial uniquely $C_{4}^{+}$-saturated graph with  one triangle,  and there are no $C_{4}^{+}$-saturated graphs  with  two, three  or  four  triangles.

\begin{theorem}\label{theorem:3.1}
 $G$  is a   uniquely $C_{4}^{+}$-saturated graph on $n$ vertices with  $g(G)=4$  if  and  only if   $G$  is a strongly regular graph  with   parameters  $(n,k,0,2)$ for some $k$.
\end{theorem}
\begin{proof}
Let  $G$  be a   uniquely $C_{4}^{+}$-saturated graph  with  $g(G)=4$.  Firstly, we show that  $G$  is  regular.  Let  $u$,  $v$  be adjacent vertices of  $G$ with  $d(u)\leq d(v)$  and  $N(v)=\{u, v_{1},v_{2},\ldots,v_{t}\}$.  We  can  see that  $N(v)$  is an independent set  since  $G$  is  triangle-free.  For  any  $i\in[t]$, $G+uv_{i}$  has exactly  one  $C_{4}^{+}$  and  $uv_{i}$  is an  \Rmnum{1}-type $C_{4}^{+}$-saturating edge  since  $g(G)=4$.  Moreover,  $v\in C_{4}^{+}$  by  \autoref{lemma:2.2}.  For convenience,  we  say  that  the  vertex set of $C_{4}^{+}$  in  $G+uv_{i}$  is $\{u,v,v_{i},w_{i}\}$  for  $i\in[t]$  and  we have $G[\{u,v,v_{i},w_{i}\}]=C_{4}$,   $w_{i}$  is  adjacent to $u$, $v_{i}$.
We can claim  $w_{i}\neq w_{j}$  for  any  $i\neq j$.  If  not,   without loss of generality,  we  say  $w_{1}=w_{2}$,  but  then $w_1$ and $v$ have three common neighbors $u,v_1,v_2$, forming a $K_{2,3}$,  which  contradicts to \autoref{lemma:2.3}.   Thus  we  can conclude that adjacent vertices in $G$ have the same degree.  Since $G$ is connected  by \autoref{lemma:2.1},  it follows that $G$ is $k$-regular for some $k$.

For  any  two  non-adjacent vertices   $u$  and  $v$  of  $G$,  then   $d_{G}(u,v)=2$  by \autoref{lemma:2.1}.  If  the  number of paths of length $2$ between $u$ and $v$  is $1$, then $G$   contains  a   triangle,   contradicting  $g(G)=4$.     Thus  the number of  such paths   is  $2$   by   \autoref{lemma:2.2},   which  means  that  each pair of non-adjacent vertices of  $G$  has exactly two  common  neighbors.  Since $G$ is  triangle-free,   each pair of adjacent vertices of  $G$  has no  common  neighbor.  And  $G$ is $k$-regular,   then  $G$  is a strongly regular graph  with   parameters  $(n,k,0,2)$  by  its  definition.

On the other hand, if  $G$  is a strongly regular graph  with   parameters  $(n,k,0,2)$,   by  its definition,    $G$  is triangle-free, thus $C_{4}^{+}$-free. And for  any non-edge  $uv$  of  $G$,   $u$  and  $v$  have exactly  two  common neighbors,  then  $G[N_{G}[u]\cap N_{G}[v]]=C_{4}$.  Thus  $G+uv$  contains exactly  one  $C_{4}^{+}$ and   $uv$  is  an  \Rmnum{1}-type  $C_{4}^{+}$-saturating  edge,  which  means that  $G$  is a nontrivial  uniquely $C_{4}^{+}$-saturated graph  with  $g(G)=4$.
\end{proof}

\begin{remark}\label{remark:3.1}
We  can see  that  only three  uniquely $C_{4}^{+}$-saturated graphs with  $g(G)=4$  are known   by  \autoref{theorem:3.1},  they  are cycle $C_{4}$  with  parameters $(4, 2, 0, 2)$,   folded $5$-cube with parameters $(16,5,0,2)$   and  Gewirtz graph with parameters $(56,10,0,2)$.
\end{remark}

We now consider  uniquely $C_{4}^{+}$-saturated graphs with $g(G)=3$. The next lemma gives  a structural decomposition.

\begin{lemma}\label{lemma:3.1}
Let $S$  be the vertex set of a triangle in a graph $G$ with $S=\{v_{1},v_{2},v_{3}\}$.  For  $i\in[3]$,  let  $V_{i}=N(v_{i})-S$.  If $G$ is nontrivial uniquely $C_{4}^{+}$-saturated, then $G$ has the following structure:

(\rmnum{1}).  $V(G)$  has a partition $V(G)=S\cup N(S)\cup N^{2}(S)$.

(\rmnum{2}).  $V_{i}\cap V_{j}=\emptyset$  for  $i\neq j$.

(\rmnum{3}).  $N(S)=V_{1}\cup V_{2}\cup V_{3}$  is an  independent set  in $G$.

(\rmnum{4}).  each vertex in $N^{2}(S)$ has  one  or two neighbors in each $V_{i}$   for  $i\in[3]$.
\end{lemma}
 \begin{proof}
By  \autoref{lemma:2.1},  $diam(G)=2$ implies (\rmnum{1}).  Figure 2 makes the structure of $G$ clearer.  Since  $G$  is  $C_{4}^{+}$-free,  (\rmnum{2})  holds.  By  \autoref{lemma:2.3},  $G$  contains no  bowknot  graph  and  house graph  as  subgraphs,  thus $N(S)$  is independent, implying (\rmnum{3}).
For  any  $u\in N^{2}(S)$  and  $i\in[3]$,  $G+uv_{i}$   has exactly one $C_{4}^{+}$.  By  \autoref{lemma:2.2},  $u$  must  have a neighbor in  $V_{i}$  for  $i\in[3]$.  If there  exists a  vertex $u\in N^{2}(S)$  which  has  at  least  three  neighbors in some $V_{i}$,   then  $u$ and $v_i$ have at least three common neighbors, forming a $K_{2,3}$, hence contradicting \autoref{lemma:2.3}.   Thus  each vertex in $N^{2}(S)$ has  one  or two neighbors in each $V_{i}$   for  $i\in[3]$, proving (\rmnum{4}).
\end{proof}

\begin{center}
\begin{tikzpicture} [inner sep=0.8mm]
\tikzstyle{place}=[circle,draw=black!,fill=black!]
\tikzstyle{r}=[circle,draw=red!,fill=red!]

\node at (0,0)    [place] {};
\node at (0,4)    [place] {};
\node at (1.5,2)    [place] {};

\draw[black,line width=1] (0,0) -- (0,4);
\draw[black,line width=1] (0,0) -- (1.5,2);
\draw[black,line width=1] (0,4) -- (1.5,2);

\coordinate [label=below left:$v_{3}$] () at (0,0);
\coordinate [label=above left:$v_{1}$] () at (0,4);
\coordinate [label=below right:$v_{2}$] () at (1.5,2);
\node at (0.5,-2.5) {$S$};

\draw[black,line width=1] (4.5,5) circle (0.8cm);
\draw[black,line width=1] (4.5,2) circle (0.8cm);
\draw[black,line width=1] (4.5,-1) circle (0.8cm);

\node at (4.5,5) {$V_{1}$};
\node at (4.5,2) {$V_{2}$};
\node at (4.5,-1) {$V_{3}$};
\node at (4.5,-2.5) {$N(S)$};

\draw[black,line width=1] (0,4) -- (3.7,5);
\draw[black,line width=1] (1.5,2) -- (3.7,2);
\draw[black,line width=1] (0,0) -- (3.7,-1);

\draw[black,line width=1] (8,2) ellipse (35pt and 80pt);
\node at (8,-2.5) {$N^{2}(S)$};
\node[inner sep=0.5mm] at (8,3)    [place] {};

\node[inner sep=0.5mm] at (4.5,5.5)    [place] {};
\node[inner sep=0.5mm] at (4.8,2.5)    [place] {};
\node[inner sep=0.5mm] at (4.8,1.5)    [place] {};
\node[inner sep=0.5mm] at (4.5,-0.5)    [place] {};

\draw[black,line width=1] (8,3) -- (4.5,5.5);
\draw[black,line width=1] (8,3) -- (4.8,2.5);
\draw[black,line width=1] (8,3) -- (4.8,1.5);
\draw[black,line width=1] (8,3) -- (4.5,-0.5);

\end{tikzpicture}

\end{center}
\begin{center}
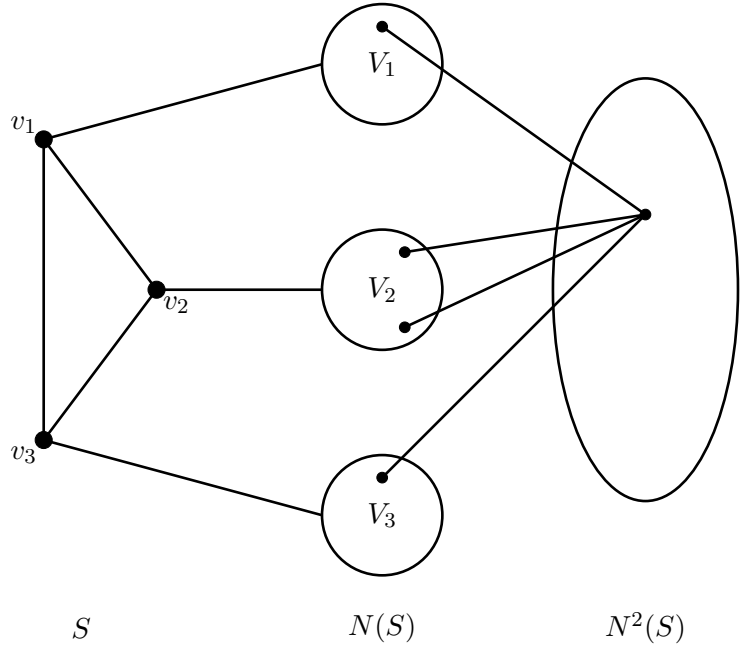
 Figure 2 : Structure of uniquely $C_{4}^{+}$-saturated graphs with girth  $3$.
\end{center}

\begin{theorem}\label{theorem:3.2}
Let $G$  be a nontrivial uniquely $C_{4}^{+}$-saturated graph with  $g(G)=3$  and $S$  be the vertex set of a triangle  in $G$.    If  $N^{2}(S)=\emptyset$,  then $G$ is isomorphic to   $C_{3}^{*}$.
\end{theorem}
 \begin{proof}
Let  $S=\{v_{1}, v_{2}, v_{3}\}$  and  $V_{i}=N(v_{i})-S$.  We have $V(G)=S\cup N(S)=S\cup V_{1}\cup V_{2}\cup V_{3}$  since  $N^{2}(S)=\emptyset$.   For  $x\in V_{i}$  and  $y\in V_{j}$  with  $i\neq j$,   $G+xy$  has exactly  one  $C_{4}^{+}$  as a subgraph. This implies that $x$ and $y$ have a common neighbor, but that vertex cannot be in $S$, nor in $N(S)$ by \autoref{lemma:3.1}, a contradiction. Therefore, only one set in  $\{V_{1}, V_{2}, V_{3}\}$ is not  empty.   Without loss of generality,  we  say  $V_{1}\neq\emptyset$  and  $V_{2}=V_{3}=\emptyset$.    If  $|V_{1}|\geq2$,  let  $\{x,y\}\subset V_{1}$,  $G+xy$  has exactly  one  $C_{4}^{+}$.  Then  $v_{1}\in C_{4}^{+}$   by  \autoref{lemma:2.2}.  We  say  $V(C_{4}^{+})=\{x,y,z,v_{1}\}$, then  $x$  or  $y$  is adjacent to $z$.  If  $z=v_{2}$  or  $v_{3}$,  then  $G$  has a  $C_{4}^{+}$  as a clear  contradiction.   If  $z\in V_{1}$,  then  $N(S)=V_{1}$  is not  independent  as a contradiction  to  \autoref{lemma:3.1}.
Thus  $|V_{1}|=1$  and   our result follows immediately.
\end{proof}

Now  we  assume   $N^{2}(S)\neq\emptyset$  for  any triangle $S$  in  $G$.

\begin{lemma}\label{lemma:3.2}
Let  $G$  be a uniquely $C_{4}^{+}$-saturated graph with  $k\geq1$ triangle and  $A$  be the vertex set  of   all triangle(s)  in  $G$.   If  $N^{2}(A)\neq\emptyset$,   then  each vertex in $G-A$  has the same degree  in $G$.
\end{lemma}
 \begin{proof}
All triangles in  $A$  are vertex-disjoint  since  $G$  has   no $C_{4}^{+}$  and  bowknot graph as  subgraphs.  Let $xy$  be a non-edge  in  $G-A$,   if   $xy$   is a \Rmnum{2}-type $C_{4}^{+}$-saturating edge,   then  there  exists a vertex in a triangle which does not belong to   $A$,  a contradiction.  Thus   $xy$ is an  \Rmnum{1}-type $C_{4}^{+}$-saturating edge.
Firstly we  claim that  adjacent vertices  in $G-A$  have the same degree  in $G$.

\begin{claim}
$d_{G}(u)=d_{G}(v)$  for  any two adjacent vertices  $u, v\in G-A$.
\end{claim}

\begin{proof}[Proof of Claim]
For  any two adjacent vertices  $u, v\in G-A$,   without loss of generality,  we say  $d_{G}(v)\leq d_{G}(u)$  and
let  $N(u)=\{v, u_{1},u_{2},\ldots,u_{t}\}$.  Then  $N(u)$  is an independent set  since $u\notin A$.  For  any  $i\in[t]$, $G+vu_{i}$  has exactly  one  $C_{4}^{+}$  and  $vu_{i}$  is an  \Rmnum{1}-type  $C_{4}^{+}$-saturating edge. Indeed, if $vu_{i}$  is a  \Rmnum{2}-type  $C_{4}^{+}$-saturating edge, then $u_i\in A$  and $v$ is adjacent to another vertex of the triangle $S$ containing $u_i$. But then $u,v \in N(S)$, thus $N(S)$ is not an independent set, contradicting \autoref{lemma:3.1}.

 Therefore,  there  exists a vertex $w_{i}$  such that $V(C_{4}^{+})=\{v,u_{i},u,w_{i}\}$,   $v, u_{i}\in N(w_i)$   and  $w_{i}\notin N(u)$  for  any $i\in[t]$.  Moreover,   we  say   $w_{i}\neq w_{j}$  for  any  $i\neq j$.  If  not,   without loss of generality,  we  say  $w_{1}=w_{2}$,   but then $w_{1}$  and  $u$  have three common neighbors  $v$, $u_{1}$, $u_{2}$,  which formed a $K_{2,3}$,   contradicting \autoref{lemma:2.3}.  Then we  have  $d_{G}(v)\geq t+1=d_{G}(u)$,  which  means $d_{G}(u)=d_{G}(v)$.   Thus   adjacent vertices in  $G-A$  have the same degree  in $G$.  This completes the proof of the claim.
\end{proof}

Let us return to the proof of the lemma. If  $N^{2}(A)\neq\emptyset$,   we  show  that each vertex in $N^{2}(A)$  has the same degree  in $G$.  It is clear if $G[N^{2}(A)]$   is  connected  since  $G[N^{2}(A)]\subseteq G-A$.  Now  we  assume  that  $G[N^{2}(A)]$  is  disconnected.  Let $u$ and $v$ be any two vertices  which  belong to any two components of $G[N^{2}(A)]$, respectively.   Then  $uv$ is  an  \Rmnum{1}-type  $C_{4}^{+}$-saturating edge and  the  two common  neighbors  of  $u$ and $v$,  say  $x$  and  $y$,   are in  $N(A)$.   Thus  $d_{G}(u)=d_{G}(x)=d_{G}(v)$,  it follows that each vertex in $N^{2}(A)$  has the same degree  in $G$.
 Moreover,  we  claim  that  each vertex in $G-A$  has the same degree  in $G$.  Since  adjacent vertices  $v\in N^{2}(A)$,  $u\in N(A)$  have  the  same degree in  $G$,  let $w$  be a vertex in $N(A)$  which  has  no neighbor  in  $N^{2}(A)$,    it  suffices to show  that  $w$  has  the  same degree in  $G$ with  the vertex in  $N^{2}(A)$.   For  any  $v\in N^{2}(A)$,  $G+wv$  has exactly  one  $C_{4}^{+}$  and  $uv$  is an  \Rmnum{1}-type  $C_{4}^{+}$-saturating edge.  Then   $w$  and  $v$  have  exactly  two  common neighbors  $x$, $y$  and  $x, y\in N(A)$.  Thus  $d_{G}(w)=d_{G}(x)=d_{G}(v)$.  This completes the proof.
\end{proof}

\begin{lemma}\label{lemma:3.3}
Let  $G$  be a uniquely $C_{4}^{+}$-saturated graph  on  $n$  vertices  with  $k\geq1$ triangle and  $A$  be the vertex set  of   all triangle(s)  in  $G$.  If  $N^{2}(A)\neq\emptyset$,  then $n=\frac{1}{2}t^{2}+\frac{1}{2}t+1$  where  $t$  is the degree in  $G$  of  vertex in $G-A$.
\end{lemma}
 \begin{proof}
Let   $S_{1}, S_{2},\ldots,S_{k}$ be the triangles in $G$  with  $V(S_{i})=\{v_{i1},v_{i2},v_{i3}\}$.   All triangles  in  $A$  are vertex-disjoint  since  $G$  has   no $C_{4}^{+}$  and  bowknot graph as  subgraphs.  If  $N^{2}(A)\neq\emptyset$,  then each vertex in $G-A$  has the same degree $t$  in $G$  by  \autoref{lemma:3.2}.
For  a  vertex $u\in N^{2}(A)$,  we have that   $A\subseteq N^{2}(u)$.   Let  $N(u)=\{u_{1}, u_{2},\ldots,u_{t}\}$  and  $N^{2}(u)=A\cup X$ with   $|X|=n-t-3k-1$.   $N(u)$  is  independent  since $u\not\in A$.  For  any  $v\in N^{2}(u)$,   $G+uv$  has exactly one $C_{4}^{+}$  and  $uv$ is an \Rmnum{1}-type  $C_{4}^{+}$-saturating edge.  Then  each  vertex in  $X\cup A$  has exactly  two  neighbors in $N(u)$.  Thus  $e[N(u), X]=2|X|=2(n-t-3k-1)$.
Let  $U_{1}=N(u)\cap N(A)$  and  $U_{2}=N(u)\backslash U_{1}$.   Then,  for  any $u_{i}\in U_{2}$,  $u_{i}$   has exactly  $t-1$  neighbors in $X$  since $u_{i}$ has degree $t$  in $G$  and  has no neighbor in $A\cup N(u)$.
For $v_{ij}\in A$,    $G+uv_{ij}$  has exactly one $C_{4}^{+}$  and  $uv_{ij}$ is an \Rmnum{1}-type $C_{4}^{+}$-saturating edge,  then  $v_{ij}$ has exactly  two  neighbors in  $U_{1}$.   Thus    $\sum\limits_{x\in A}d_{U_{1}}(x)=2|A|=6k$.  Moreover,  $\sum\limits_{x\in A}d_{U_{1}}(x)=e[A, U_{1}]=\sum\limits_{u_{i}\in U_{1}}d_{A}(u_{i})=6k$.  Since  each vertex in $G-A$  has the same degree $t$  in $G$,  we have $\sum\limits_{u_{i}\in U_{1}}d_{X}(u_{i})=e[U_{1}, X]=t|U_{1}|-6k-|U_{1}|$  and
$\sum\limits_{u_{i}\in U_{2}}d_{X}(u_{i})=e[U_{2}, X]=|U_2|(t-1)=(t-|U_{1}|)(t-1)$.
We have  $2|X|=e[N(u), X]=e[U_{1}, X]+e[U_{2}, X]$,  which implies  $2(n-t-3k-1)=(t|U_{1}|-6k-|U_{1}|)+(t-|U_{1}|)(t-1)$.  Therefore,  $n=\frac{1}{2}t^{2}+\frac{1}{2}t+1$.
 \end{proof}

\begin{theorem}\label{theorem:3.3}
Let  $G$  be  a uniquely  $C_{4}^{+}$-saturated graph   with  $k\geq1$  triangle  and  $A$  be the vertex set  of   all triangle(s)  in  $G$.  Then   $N^{2}(A)=\emptyset$.
\end{theorem}
 \begin{proof}
Let   $A$  be the vertex set  of   all triangle(s)  in  $G$.   All triangles  in  $A$  are vertex-disjoint  since  $G$  has   no $C_{4}^{+}$  and  bowknot graph as  subgraphs.  Let  $S$  be the vertex set of  a triangle  in  $G$  with  $S=\{v_{1}, v_{2}, v_{3}\}$  and   $V_{i}=N(v_{i})-S$   with  $|V_{i}|=s_{i}$   for  $i\in[3]$.   By  \autoref{lemma:3.1},  $V(G)=S\cup N(S)\cup N^{2}(S)$  is a  partition  of  $V(G)$,   $V_{i}\cap V_{j}=\emptyset$  for  $i\neq j$  and  $N(S)=V_{1}\cup V_{2}\cup V_{3}$  is an  independent set  in $G$.  Moreover,    each vertex in $N^{2}(S)$  has  one  or two neighbors in each $V_{i}$   for  $i\in[3]$  by  \autoref{lemma:3.1}.
For  $i\in[3]$  and  $j\in[2]$,   we   denote  $N_{i}^{j}=\{u\in N^{2}(S): |N(u)\cap V_{i}|=j\}$,   then  $N^{2}(S)=N_{1}^{1}\cup N_{1}^{2}=N_{2}^{1}\cup N_{2}^{2}=N_{3}^{1}\cup N_{3}^{2}$.      Thus,  for $i\in[3]$,   we  have $|N_{i}^{2}|\leq\binom{s_{i}}{2}$.

Let  $V_{1}=\{x_{1}, x_{2},\ldots,x_{s_{1}}\}$.  For  any  non-edge $x_{i}x_{j}$,  we claim that  $x_{i}x_{j}$ is an \Rmnum{1}-type $C_{4}^{+}$-saturating edge.   If  not,  let  $x_{1}x_{2}$  be a \Rmnum{2}-type $C_{4}^{+}$-saturating edge   and   $V(C_{4}^{+})=\{x_{1}, x_{2}, v_{1}, x_{i}\}$,  then  $x_{1}x_{i}$  or  $x_{2}x_{i}$  is an edge  in  $V_{1}$  as  a contradiction.  Thus  there  are  $\binom{s_{1}}{2}$  \Rmnum{1}-type $C_{4}^{+}$-saturating edges  $\{x_{i}x_{j}: i\neq j\}$   in  $V_{1}$  and    $v_{1}\in V(C_{4}^{+})$  in  $G+x_{i}x_{j}$  is a common neighbor of  $x_{i}$  and $x_{j}$.   Since  each  vertex in  $V_{1}$  has no neighbor in  $N(S)\cup\{v_{2}, v_{3}\}$  by  \autoref{lemma:3.1},  we  have that,  for  any  $i\neq j$,  there exists exactly  one vertex in $N^{2}(S)$   such  that  it  belongs  to  $V(C_{4}^{+})$  in  $G+x_{i}x_{j}$  and is another  common neighbor of  $x_{i}$  and $x_{j}$ in addition to $v_{1}$,  which  means  $|N_{1}^{2}|\geq\binom{s_{1}}{2}$.   Thus  $|N_{1}^{2}|=\binom{s_{1}}{2}$  and
$e[x, N^{2}(S)]\geq s_{1}-1$  for  any  $x\in V_{1}$,   equality  holds if and only if  $x$  has no neighbor in $N_{1}^{1}$.   Similarly,  we have  that,  for  $i=2, 3$,   $|N_{i}^{2}|=\binom{s_{i}}{2}$  and
$e[x, N^{2}(S)]\geq s_{i}-1$  for  any  $x\in V_{i}$,   equality  holds if and only if  $x$  has no neighbor in $N_{i}^{1}$.

Let  $V_{1}'=\{x\in V_{1}: N(x)\cap N_{1}^{1}\neq\emptyset\}$,  we claim  that,  for  a vertex  $x$  in  $V_{1}$,   $x\in V_{1}'$  if  and  only if  $x\in K_{3}$  in $G$.  If  $x\in V_{1}'$,  there exists $u\in N_{1}^{1}$  adjacent  to  $x$.   $G+v_{1}u$  has  exactly one $C_{4}^{+}$  and   $v_{1}u$ is  a  \Rmnum{2}-type $C_{4}^{+}$-saturating edge since $u$ has only one common neighbor with $v$.  Then  there  exists $u'\in N^{2}(S)$  adjacent to $u$,  $x$  and  $V(C_{4}^{+})=\{v_{1}, u, x,  u'\}$,   which  means  $x\in K_{3}$.   On the other hand,   if  $x\in V_{1}$  belongs  to a $K_{3}$  in $G$,  then  there exist vertices $u_{1}$,  $u_{2}$  in  $N^{2}(S)$   with  $V(K_{3})=\{x, u_{1}, u_{2}\}$  since  $N(S)$  is   independent.  Moreover,   $u_{1}$,  $u_{2}\in N_{1}^{1}$.  Otherwise,  say $u_{1}$  has  another neighbor  $x'$  in  $V_{1}$  in addition to  $x$.  But $G+xx'$  has two $C_{4}^{+}$  with  vertex  sets  $\{x, x', v_{1}, u_{1}\}$  and  $\{x, x', u_{1}, u_{2}\}$   as a  contradiction.  Thus  $u_{1}$,  $u_{2}\in N_{1}^{1}$,  which  implies   $x\in V_{1}'$.  Therefore,  for  a vertex $x$  in  $V_{1}$,   we have  $x\in V_{1}'$  if  and  only if  $x\in K_{3}$  in $G$.   Similarly,  we  denote  $V_{2}'=\{y\in V_{2}: N(y)\cap N_{2}^{1}\neq\emptyset\}$  and  $V_{3}'=\{z\in V_{3}: N(z)\cap N_{3}^{1}\neq\emptyset\}$.  Then  for  a vertex  $y$ in  $V_{2}$,    $y\in V_{2}'$  if  and  only if  $y\in K_{3}$  in $G$   and   for  a vertex  $z$ in  $V_{3}$,    $z\in V_{3}'$  if  and  only if  $z\in K_{3}$  in $G$.

If  $N^{2}(A)\neq\emptyset$,  let  $u\in N^{2}(A)$,  then  $u\in N^{2}(S)$.  We  say $x\in N(u)\cap V_{1}$,  $y\in N(u)\cap V_{2}$  and  $z\in N(u)\cap V_{3}$.  Moreover,  $x$, $y$  and $z$  are not in any $K_{3}$  in $G$  since   $u\in N^{2}(A)$.  Then  we have   $x\notin V_{1}'$,  $y\notin V_{2}'$   and  $z\notin V_{3}'$,   which  implies  $e[x, N^{2}(S)]=s_{1}-1$,  $e[y, N^{2}(S)]=s_{2}-1$  and  $e[z, N^{2}(S)]=s_{3}-1$.   Since  $x\in G-A$,  $x$  has  degree $t$  in $G$   by  \autoref{lemma:3.2}.  Then  $t=d_{G}(x)=d_{N^{2}(S)}(x)+d_{S}(x)=s_{1}$   since  $N(S)$  is   independent,  which  means $s_{1}=t$.      Similarly,  we have  $s_{2}=t$  and  $s_{3}=t$.  Then   $n=|S|+|V_1|+|V_2|+|V_3|+|N_{1}^{2}|+|N_{1}^{1}|=3+3t+\binom{t}{2}+|N_{1}^{1}|=\frac{1}{2}t^{2}+\frac{5}{2}t+3+|N_{1}^{1}|>\frac{1}{2}t^{2}+\frac{1}{2}t+1$,  contradicting  \autoref{lemma:3.3}.   Thus  $N^{2}(A)=\emptyset$.
\end{proof}

\begin{theorem}\label{theorem:3.4}
Let  $G$  be a uniquely $C_{4}^{+}$-saturated graph  on  $n$  vertices  with  $k\geq2$ triangles  and  $A$  be the vertex set  of   all triangles in $G$.  Let $t=min\{k-|N_{G}(u)\cap A|: u\in N(A)\}$,  then $n\leq\frac{1}{2}k^2+\frac{10t+5}{2}k+\frac{25t^2+7t}{2}+1$.
\end{theorem}
\begin{proof}
Let  $G$  be a uniquely $C_{4}^{+}$-saturated graph  on  $n$  vertices  with  $k\geq2$ triangles $S_{1}, S_{2},\ldots,S_{k}$ with $V(S_{i})=\{v_{i1},v_{i2},v_{i3}\}$  and  $A=\bigcup\limits_{i=1}^{k}V(S_{i})$.  Let  $V_{ij}=N_{G}(v_{ij})-A$  for  $i\in[k]$  and  $j\in[3]$, then  $N(A)=\bigcup\limits_{i=1}^{k}\big(\bigcup\limits_{j=1}^{3}V_{ij}\big)$.   By  \autoref{theorem:3.3},  $V(G)$  has a  partition  $V(G)=A\cup N(A)$.  Since  $G$  has  no $C_{4}^{+}$,  any vertex in $N(A)$  has  at most $k$  neighbors in $A$,  that is,  any vertex in $N(A)$  has  at most one  neighbor in  each  triangle in $A$.
Since  $t=min\{k-|N_{G}(u)\cap A|: u\in N(A)\}$,  there is a vertex $u\in N(A)$  with $k-t$  neighbors in $A$.  Moreover, let $N_{G}(u)\cap A=\{v_{(t+1)1},v_{(t+2)1},\ldots,v_{k1}\}$.  For   $i\in[t]$ and  $j\in[3]$,   we  can see  that  $u$  has exactly two neighbors in $V_{ij}$  since  $uv_{ij}$  is an \Rmnum{1}-type  $C_{4}^{+}$-saturating edge.  For  $t+1\leq i\leq k$  and  $j\in[3]$,  we claim that $u$  has no neighbor in $V_{ij}$.  If  not,  say $u$  is adjacent to a vertex in $V_{(t+1)j}$,  then  $G$  has a  bowknot graph   or  house graph,  contradicting  \autoref{lemma:2.3}.  Thus  $|N(A)\cap N_{G}(u)|\leq6t$,  which  means $|N_{G}(u)|\leq(k-t)+6t=k+5t$.  For  any vertex $u'\in N(A)\backslash N_{G}(u)$,   we  have that  $uu'$ is an \Rmnum{1}-type $C_{4}^{+}$-saturating edge,  thus  $u'$   has exactly  two  neighbors in $N_{G}(u)$. Two vertices of $N(A)\backslash N_{G}(u)$ cannot have the same two neighbors in $N_{G}(u)$, as those vertices with $u$ would form a $K_{2,3}$, contradicting \autoref{lemma:2.3}. Therefore,
$|N(A)\backslash N_{G}(u)|\leq\binom{k+5t}{2}$,  which means $n=3k+|N(A)|\leq3k+1+6t+\binom{k+5t}{2}=\frac{1}{2}k^2+\frac{10t+5}{2}k+\frac{25t^2+7t}{2}+1$.  Our result follows immediately.
\end{proof}

\begin{remark}\label{remark:3.2}
For  \autoref{theorem:3.4},  the upper bound of order $n$  is increasing with respect to $t$.  If $t=0$,    then  there exists a vertex $u\in N(A)$  having $k$  neighbors in $A$.  Moreover,  we have $d_{G}(u)=k$.  And  the upper bound of order $n$  is $\frac{1}{2}k^2+\frac{5}{2}k+1$.  In the opposite side,  if  $t=k-1$,  then any vertex in $N(A)$  has exactly one neighbor in $A$.   The upper bound of order $n$  is $18k^2-24k+10$.
\end{remark}

By   \autoref{theorem:3.4} and   \autoref{remark:3.2},  we  can  get that  there are no uniquely $C_{4}^{+}$-saturated graphs with  $k\geq2$ triangles and large  orders.

\begin{theorem}\label{theorem:3.5}
For $n>18k^2-24k+10$  with  $k\geq2$,   there are no uniquely $C_{4}^{+}$-saturated graphs  on  $n$  vertices  with  $k$ triangles.
\end{theorem}

By   \autoref{theorem:3.2} and   \autoref{theorem:3.3},  we  can  get the following  result about uniquely $C_{4}^{+}$-saturated graphs with  one  triangle immediately.

\begin{theorem}\label{theorem:3.6}
$C_{3}^{*}$ is the  only nontrivial  uniquely  $C_{4}^{+}$-saturated graph with  one  triangle.
\end{theorem}

There exists  exactly  one  nontrivial  uniquely $C_{4}^{+}$-saturated graph  with  one  triangle.  In the following results,   we prove that there  are no  uniquely  $C_{4}^{+}$-saturated graphs  with  two, three  or four triangles.

\begin{theorem}\label{theorem:3.7}
There  are no uniquely  $C_{4}^{+}$-saturated graphs  with  two  or  three  triangles.
\end{theorem}
 \begin{proof}
Let  $G$  be  a uniquely  $C_{4}^{+}$-saturated graph   with   two  triangles  $S_{1}=\{v_{11}, v_{12}, v_{13}\}$, $S_{2}=\{v_{21}, v_{22}, v_{23}\}$   and  $A=V(S_{1})\cup V(S_{2})$.  By  \autoref{theorem:3.3},  $V(G)$  has a partition
$V(G)=A\cup N(A)$.  For   any  $u\in N(A)$,  $u$  has  at  most one  neighbor in  $S_{1}$  and  $S_{2}$  since  $G$  is  $C_{4}^{+}$-free.   We  have  $e[S_{1}, S_{2}]\leq1$  since  $G$  has  no $C_{4}^{+}$  and  house graph  as subgraphs. Without loss of generality,  we  say  $E[S_{1}, S_{2}]=\{v_{11}v_{21}\}$   if  $e[S_{1}, S_{2}]=1$.
There  exist  at least  four non-edges $\{v_{12}v_{22},v_{12}v_{23},v_{13}v_{22},v_{13}v_{23}\}$  between  $S_{1}$  and  $S_{2}$  as  \Rmnum{1}-type $C_{4}^{+}$-saturating edges.   Let   $V(C_{4}^{+})=\{v_{12}, v_{22},x_{1},x_{2}\}$  in  $G+v_{12}v_{22}$   and  $V(C_{4}^{+})=\{v_{12}, v_{23},x_{3},x_{4}\}$  in  $G+v_{12}v_{23}$.  $G+x_{1}x_{3}$   has exactly one $C_{4}^{+}$  and $x_{1}x_{3}$ is an \Rmnum{1}-type $C_{4}^{+}$-saturating edge  since $G$ has exactly two triangles $S_{1},S_{2}$.   Let $\{x_{1}, x_{3},v_{12},x\}$ be the vertex set of the $C_{4}^{+}$ in $G+x_{1}x_{3}$, then $x$  is  adjacent to $x_{1}, x_{3}$. Since $x\neq v_{12}$ has a common neighbor with $v_{12}$, $x$ cannot be in $S_{1}$. Since $x$ has common neighbor $x_1$ with $v_{22}$ and common neighbor $x_3$ with $v_{23}$, $x$ cannot be in $S_{2}$, thus
$x\in N(A)$.  However, $x$ has a neighbor in $N(S_{1})$ and a neighbor in $N(S_{2})$, thus  $N(S_{1})$ or $N(S_{2})$ is not an independent set,   contradicting   \autoref{lemma:3.1}.

Let  $G$  be now  a uniquely  $C_{4}^{+}$-saturated graph   with  three  triangles  $S_{1}=\{v_{11}, v_{12}, v_{13}\}$, $S_{2}=\{v_{21}, v_{22}, v_{23}\}$,  $S_{3}=\{v_{31}, v_{32}, v_{33}\}$   and  $A=V(S_{1})\cup V(S_{2})\cup V(S_{3})$.  By  \autoref{theorem:3.3},  $V(G)=A\cup N(A)$. For   any  $u\in N(A)$,  $u$  has  at  most one  neighbor in  $S_{1}$,  $S_{2}$  and  $S_{3}$  since  $G$  is  $C_{4}^{+}$-free.   We  have $e[S_{i}, S_{j}]\leq1$  for  any  $i\neq j$  since  $G$  has  no $C_{4}^{+}$  and  house graph  as subgraphs.  Without loss of generality,  we  say  that  $v_{11}\in S_{1}$  has no neighbor  in  $S_{2}\cup S_{3}$,  $v_{21}\in S_{2}$  has no neighbor  in  $S_{1}\cup S_{3}$  and  $v_{31}\in S_{3}$  has no neighbor  in  $S_{1}\cup S_{2}$.     Then  $G+v_{11}v_{21}$,  $G+v_{11}v_{31}$    have  exactly  one  $C_{4}^{+}$  and   $v_{11}v_{21}$,  $v_{11}v_{31}$  are   two \Rmnum{1}-type $C_{4}^{+}$-saturating edges,   which  implies that $v_{11}$,  $v_{21}$  have  exactly  two  common  neighbors,  say  $x_{1}$, $x_{2}$,  in  $N(A)$  and   $v_{11}$,  $v_{31}$  have  exactly  two  common  neighbors,  say  $y_{1}$, $y_{2}$,  in  $N(A)$.  Without loss of generality,   we can assume  $x_1\neq y_2$.    If not,  $x_{1}$ and $x_{2}$  are also the common neighbors of $v_{11}$ and $v_{31}$,  but  then  $x_{1}$ and $x_{2}$  have three common neighbors $v_{11}$, $v_{21}$, $v_{31}$,  forming a $K_{2,3}$, which contradicts to \autoref{lemma:2.3}.   Observe that $x_{1}y_{2}$ is a non-edge, otherwise they would form a triangle with $v_{11}$. By the $C_{4}^{+}$-saturated property, there is a common neighbor $z$ of $x_1$ and $y_2$ different from $v_{11}$. If $z\in N(A)$, then $z\in N(S_{i})$ for some $i\in[3]$, but $N(S_i)$ contains $x_{1}$ or $y_{2}$, thus there is an edge inside $N(S_{i})$, contradicting  \autoref{lemma:3.1}. Therefore, $z\in A$. The only possibilities are $v_{21}$ and $v_{31}$, since each of $x_1$ and $y_2$  has at most one neighbor in each triangle.
If  $z=v_{21}$,  then  $y_{2}$  is the third common neighbor of $v_{11}$ and $v_{21}$ different from $x_{1}$ and  $x_{2}$  and they  would  form a $K_{2,3}$,    contradicting \autoref{lemma:2.3}.    Similarly,  if  $z=v_{31}$,   then   $v_{11}$ and $v_{31}$  have three  common neighbors  $y_{1}$, $y_{2}$, $x_{1}$, forming a $K_{2,3}$ again and contradicting \autoref{lemma:2.3}.
\end{proof}



\begin{theorem}
    There is no uniquely $C_4^+$-saturated graphs with four triangles.
\end{theorem}

\begin{proof}
Let  $G$  be a uniquely $C_{4}^{+}$-saturated graph  on  $n$  vertices  with  four triangles  $S_{1}, S_{2},S_{3},S_{4}$  and $A=\bigcup\limits_{i=1}^{4}V(S_{i})$  with  $V(S_{i})=\{v_{i1},v_{i2},v_{i3}\}$.  Let  $V_{ij}=N_{G}(v_{ij})-A$  for  $i\in[4]$  and  $j\in[3]$,  then $N(A)=\bigcup\limits_{i=1}^{4}\big(\bigcup\limits_{j=1}^{3}V_{ij}\big)$.  By  \autoref{theorem:3.3},  $V(G)$  has a partition $V(G)=A\cup N(A)$.
Let $m$ be the largest integer such that there is a vertex outside $A$ with $m$ neighbors in $A$.
Firstly  we  claim  $m\geq2$.  If $m=1$,  then  any vertex in $N(A)$  has exactly  one neighbor in $A$.  Since there exists at most one edge between any two triangles,   we  can get that  there are at least four \Rmnum{1}-type  $C_{4}^{+}$-saturating edges  between $S_{1}$  and  $S_{2}$.  Moreover,  the common   neighbors  of  the endpoints of these  $C_{4}^{+}$-saturating  edges  are  in $A$  since $m=1$.  But $G$  has $K_{2,3}$  as a contradiction.  We proceed by a case analysis.

\textbf{Case 1.} $m=2$.

Observe that there is at most one edge between any two triangles. Without loss of generality, either $v_{11}v_{21}$ is an edge between $S_1$ and $S_2$, or there is no edge between them. Furthermore, we can assume that from $v_{12}$ there is no edge to $S_3$ and from $v_{22}$ there is no edge to $S_4$.  These above assumptions imply that $v_{12}$ and $v_{22}$ do not have a common neighbor in $A$. Therefore, they have two common neighbors outside $A$, i.e., $V_{12}\cap V_{22}$ contains two vertices $u_1,u_2$. Analogously we can assume that $V_{32}\cap V_{42}$ contains $w_1,w_2$.  Now $N_G(u_i)\cap N_G(w_j)=\emptyset$ for any $i$ and $j$.  Otherwise, $x\in N_G(u_i)\cap N_G(w_j)$ for some $i,j$.  Then  $x\in N(A)$  since $m=2$. But $x$ has neighbors in $N(S_i)$ for each $i\le 4$, thus one of the sets $N(S_i)$ is not independent, contradicting  \autoref{lemma:3.1}.  Thus each $u_iw_j\in E(G)$. This implies that the common neighbors of $u_1$ and $u_2$ are $w_1$, $w_2$, $v_{12}$ and $v_{22}$, giving us a $K_{2,4}$ inside $G$, a contradiction.

\textbf{Case 2.} $m=3$.

Observe that the argument in Case 1 works here, giving us $u_1,u_2\in V_{12}\cap V_{22}$ and $w_1,w_2\in V_{32}\cap V_{42}$, but $u_i$ and $w_j$ may have common neighbors in $A$, namely $v_{12}$, $v_{22}$, $v_{32}$ or $v_{42}$. If two adjacent edges of the form $u_iw_j$ are in $E(G)$, then we have a $K_{2,3}$ in $G$ as in Case 1, a contradiction. Therefore,  two   parallel edges of the form $u_iw_j$ are missing from $E(G)$, without loss of generality, $u_1w_1$ and $u_2w_2$ are not in $E(G)$. Then $u_1$ and $w_1$ have two common neighbors, without loss of generality, $v_{12}$ and $v_{32}$. Then $u_2$ is not adjacent to $v_{32}$, since then $u_1$ and $u_2$ would have three common neighbors. Similarly $w_2$ is not adjacent to $v_{12}$ as then $v_{12}$ would be the third common neighbor of $w_1$ and $w_2$. Therefore, the common neighbors of $u_2$ and $w_2$ are $v_{22}$ and $v_{42}$.

Consider now a vertex $z\in V_{41}$. If $u_1z\not\in E(G)$, then $u_1$ and $z$ have two common neighbors. Since the neighbors of $u_1$ are $v_{12}$, $v_{22}$, $v_{32}$ and some vertices in $V_{41}\cup V_{42}\cup V_{43}$, the common neighbors  of $u_1$ and $z$ are among $v_{12}$, $v_{22}$ and $v_{32}$. But any two of these already have two common neighbors among $u_1,u_2,w_1,w_2$, thus with $z$ they form a $K_{2,3}$, a contradiction. Therefore, $u_1z\in E(G)$. This implies that $|V_{41}|\le 2$, since $u_1$ has two common neighbors with $v_{41}$.
Moreover, the vertices in $V_{41}$ are not adjacent to any other vertex of $A$.  Otherwise, $G$  has $C_{4}^{+}$ or bowknot graph or  house graph  as a contradiction  since the neighbor $u_{1}$ of them is adjacent to vertices in $S_1,S_2,S_3$. By symmetry, the same holds for any $V_{ij}$ with $j\neq 2$. Now recall from $v_{13}$ there is no edge to $S_2$ and from $v_{23}$ there is no edge to $S_1$. There must be two common neighbors of $v_{13}$ and $v_{23}$ in $A$.   If not,  let $x$  be a common neighbor of $v_{13}$ and $v_{23}$  outside $A$,  then $x\in V_{13}\cap V_{23}$.  From the  above  argument,  $x$  is adjacent to $w_{1}$  and  $w_{2}$,  but $G$  has a $K_{2, 3}$  with  vertex set $\{w_{1}, w_{2}, x, v_{32}, v_{42}\}$,  a contradiction.  This implies that they must be in $S_3$ and $S_4$. Consider now $v_{13}$ and $v_{22}$.  Similarly as above,  since  $G$  has no bowknot graph, there  must have two common neighbors of $v_{13}$ and $v_{22}$ inside $S_3$ and $S_4$, thus those two vertices both are adjacent to $v_{22}$ and $v_{23}$ giving us a $C_4^+$,  a contradiction.

\textbf{Case 3.} $m=4$.

Without loss of generality, $u\in V_{11}\cap V_{21}\cap V_{31}\cap V_{41}$. This implies that the neighbors of $u$ are $v_{11}$, $v_{21}$, $v_{31}$, $v_{41}$ and no other vertices  since $u$ has at most one neighbor in each triangle, and $u\in N(S)$ for each triangle $S$, thus has no neighbors in $N(S)$. Each of the other  $n-13$ vertices  in $N(A)\setminus\{u\}$ has two common neighbors with $u$.  Then  $|N(A)\setminus\{u\}|\leq 6$  since $N_{G}(u)=\{v_{11}, v_{21}, v_{31}, v_{41}\}$.
Note that $v_{i1}$ and $v_{j1}$  with $i\neq j$  have two common neighbors, including $u$.  Since $G$  has  no  $K_{2, 3}$, we have $|U|=6$ with $U=N(A)\setminus\{u\}$.  And  each vertex in $U$ has exactly  two neighbors  in  $\{v_{11}, v_{21}, v_{31}, v_{41}\}$.
Since $e[S_{i}, S_{j}]\leq1$ for any $i\neq j$  and  $G$  has no bowknot graph or  house graph,   $v_{11}$ has no neighbor in  $S_{2}\cup S_{3}\cup S_{4}$ and  $v_{21}$ has no neighbor in $S_{1}\cup S_{3}\cup S_{4}$.  Without loss of generality,  we say that
$E[S_{1}, S_{2}]=v_{13}v_{23}$ if $E[S_{1}, S_{2}]\neq\emptyset$.  Now consider  $G+v_{12}v_{21}$,  then $v_{12}$, $v_{21}$ have exactly  two common neighbors  which are the two vertices of $U$ in $N_G(v_{21})\cap N_G(v_{31})$  and $N_G(v_{21})\cap N_G(v_{41})$, respectively.  Thus $v_{12}$  has no neighbor in $S_{3}\cup S_{4}$ since  $G$  has no bowknot graph or  house graph.  Similarly we consider $G+v_{12}v_{22}$,  but $v_{12}$ and  $v_{22}$ have  at most one  common neighbor which is the vertex of $U$ in $N_G(v_{31})\cap N_G(v_{41})$,  then  $G+v_{12}v_{22}$  has no $C_{4}^{+}$,  a contradiction.
\end{proof}

\section{Concluding remarks}

The \textit{book graph with $p$ pages} $B_p$ consists of $p$ triangles sharing an edge. In other words, we take an edge $uv$ that we call the \textit{rootlet edge}, and add $p$ vertices joined to both $u$ and $v$. Then $C_4^+$ is $B_2$. Some of our results extend to  uniquely $B_p$-saturated graphs with $p\ge 3$.
Again, there are two types of edges of $B_p$, the rootlet edge and the other edges, called \textit{page edges}. A nontrivial uniquely $B_p$-saturated graph is connected, has diameter 2 and girth at most 4 by the same or even simpler arguments than in the case $p=2$ earlier.  In the case a $B_p$-saturated graph $G$ does not contain $B_{p-1}$, then adding a new edge $uv$ of $G$ creates a $B_p$, where $uv$ is the rootlet edge. In particular, $u$ and $v$  have $p$ common neighbors. Also, a uniquely $B_p$-saturated graph $G$ must be $K_{2,p+1}$-free, since then the edge inside the smaller part creates more than one copy of $B_p$. These observations help prove an analogue of Theorem \ref{theorem:3.1}.

\begin{theorem}
$G$ is a nontrivial uniquely $B_p$-saturated graph on $n$ vertices with $g(G)=4$ if and only if $G$ is a strongly regular graphs with parameters $(n,k,0,p)$ for some $k$.
\end{theorem}

\begin{proof} First we show that $G$ is regular. Let $u$ and $v$ be adjacent vertices with $d(u)\le d(v)$ and $N(v)=\{u,v_1,v_2,\dots, v_t\}$. Then adding $uv_i$ with  $i\in[t]$ creates a $B_p$, with $uv_i$ being the rootlet edge, thus $u$ and $v_i$ have exactly $p$ common neighbors $v,w_i^1,\dots,w_i^{p-1}$. Note that we may have that $w_i^j=w_{i'}^{j'}$.
Let $W=\{w_{i}^{j}: i\in[t], j\in[p-1]\}$,  then $|W|\leq t(p-1)$.   For any $w_{i}^{j}\in W$,  since $g(G)=4$  and $G+w_{i}^{j}v$ has exactly one copy of $B_{p}$,  $w_{i}^{j}$ has exactly $p-1$  neighbors in $N(v)\backslash\{u\}=\{v_1,v_2,\dots, v_t\}$.  And for any $v_{i}$ with  $i\in[t]$,  $v_{i}$  has exactly $p-1$ neighbors  $\{w_i^1,\dots,w_i^{p-1}\}$ in $W$.  Thus by double counting, we have that $e[W, N(v)\backslash\{u\}]=|W|(p-1)=|N(v)\backslash\{u\}|(p-1)=t(p-1)$,  which means that $|W|=t$.  Then  $d(u)\geq t+1$,  proving the regularity.
The girth condition implies that adjacent vertices have no common neighbor, and we have mentioned earlier that non-adjacent vertices have exactly $p$ common neighbors.

If $G$ is a strongly regular graph with parameters $(n,k,0,p)$ for some $k$, then it is triangle-free thus $B_p$-free, but adding any edge creates a $B_p$ with the $p$ common neighbors of its endpoints. This $B_p$ is unique, since the new edge must be a rootlet edge in any new $B_p$.
\end{proof}

Note that another analogue of Theorem \ref{theorem:3.1} is to consider $B_{p-1}$-free graphs, since in those graphs we still have that any edge added can only be the rootlet edge of any copy of $B_p$. In particular, strongly regular graphs with parameters $(n,k,q,p)$, $q\le p-2$ are also uniquely $B_p$-saturated. As an example,  the complete $r$-partite graph $K_{k,\ldots,k}$ is a $B_{(r-2)k+1}$-free  uniquely $B_{(r-1)k}$-saturated graph, which is just a strongly regular graph with parameters $(rk,(r-1)k,(r-2)k,(r-1)k)$.
 Moreover, if we  have a graph where non-adjacent vertices have exactly $p$ common  neighbors and adjacent vertices have at most $p-2$ common neighbors,  it is also uniquely $B_p$-saturated by the arguments in the above proof.  For example, let us remove from a complete graph $K_{rm}$   the edge set of $r$ vertex-disjoint copies of strongly regular graphs with parameters $(m,r,\lambda,\mu)$, where $\lambda\ge \mu$. Then two non-adjacent vertices in the obtained graph $G$ have  $rm-2r+\lambda$ common neighbors. Two adjacent vertices in $G$ from the same removed strongly regular graph have  $rm-2-2r+\mu$ common neighbors in $G$. Two vertices from distinct removed graphs must be adjacent in $G$. The number of their common neighbors is  $rm-2-2r$. Therefore, $G$ is uniquely  $B_{rm-2r+\lambda}$-saturated.  Note that $G$  is not a  strongly regular graph  but is $(rm-r-1)$-regular.

Let us consider now how $C_3^*$ is generalized to our setting. Let us take a $(p-1)$-regular graph $G_0$ that is uniquely $B_{p-1}$-saturated, add a vertex $u$ joined to everything, and another vertex $v$ joined to $u$, to obtain $G$. In the case of $C_3^*$, the 1-regular graph $G_0$  is a single edge. There are at least two ways to generalize this to larger $p$: we can choose $K_p$ or $K_{p-1,p-1}$ as $G_0$. As we have seen, balanced complete $r$-partite graph $K_{k,\ldots,k}$  with $r>2$  and  $k=\frac{p-1}{r-1}$  can also be chosen as $G_0$.
Clearly, $G$ is $B_p$-free, since  $u$  is the only vertex with degree greater than $p$.  If  we add an edge $vw$, then we have that $uw$  is the rootlet edge  of $B_{p}$ and  $u$,  $w$ have common neighbors  $v$  and the $p-1$ other neighbors of $w$. If we add an edge $ww'$ with both vertices in $G_0$, then we find a unique $B_{p-1}$ in $G_0$ that is extended to a $B_p$ with $u$.

\section*{Acknowledgments}
This work is supported by the National Natural Science Foundation of China (No. 12271251), the Fundamental Research Funds for the Central Universities (No. NC2024007), and by National Research, Development and Innovation Office - NKFIH under the grant KKP-133819.

\end{document}